\documentclass[a4paper]{amsart}
\usepackage{amsmath}
\usepackage{amssymb}

\theoremstyle{plain}
\newtheorem{theorem}{\bf Theorem}[section]
\newtheorem{conjecture}[theorem]{\bf Conjecture}
\newtheorem{proposition}[theorem]{\bf Proposition}
\newtheorem{lemma}[theorem]{\bf Lemma}

\newtheorem{corollary}[theorem]{\bf Corollary}

\newtheorem{theirtheorem}{Theorem}
\newtheorem{theirproposition}[theirtheorem]{Proposition}

\theoremstyle{definition}

\newcommand{\Z}{\mathbb Z}

\newcommand{\F}{\mathcal F}
\newcommand{\Fc}{\mathcal F}
\newcommand{\vp}{\mathsf v}

\newcommand{\D}{\mathsf D}
\newcommand{\rr}{\mathsf r}
\newcommand{\dd}{\mathsf d}

\newcommand{\ord}{\text{\rm ord}}

\newcommand{\supp}{\text{\rm supp}}

\newcommand{\h}{\mathsf{h}}

\newcommand{\be}{\begin{equation}}
\newcommand{\ee}{\end{equation}}

\newcommand{\bnml}{\begin{multline}}
\newcommand{\enml}{\end{multline}}
\newcommand{\buml}{\begin{multline*}}
\newcommand{\euml}{\end{multline*}}

\newcommand{\ber}{\begin{eqnarray}}
\newcommand{\eer}{\end{eqnarray}}
\newcommand{\nn}{\nonumber}

\newcommand{\Sum}[2]{\underset{#1}{\overset{#2}{\sum}}}
\newcommand{\Summ}[1]{\underset{#1}{\sum}}

\begin{document}

\title{Representation of finite abelian group elements by subsequence sums}
\author{D. J. Grynkiewicz $^1$}\thanks{This project was begun while the first author was supported by NSF grant DMS-0502193, and completed while he was supported by FWF project number M1014-N13}
\author{E. Marchan $^2$}
\author{O. Ordaz $^3$}

\address{$^1$ Institut f\"ur Mathematik und Wissenschaftliches Rechnen,
Karl-Franzens-Universit\"at Graz,
Heinrichstra\ss e 36,
8010 Graz, Austria.}
\address{$^2$ Departamento de Matem\'aticas. Decanato de Ciencias y Tecnolog\'{i}as,
Universidad Centroccidental Lisandro Alvarado, Barquisimeto,
Venezuela.}
\address{$^3$ Departamento de Matem\'aticas y Centro ISYS, Facultad de Ciencias,
Universidad Central de Venezuela, Ap. 47567, Caracas 1041-A,
Venezuela.}

\begin{abstract} Let $G\cong C_{n_1}\oplus \cdots \oplus C_{n_r}$ be a finite and nontrivial abelian group with $n_1|n_2|\ldots|n_r$. A conjecture of Hamidoune says that if $W=w_1\cdots w_n$ is a sequence of integers, all but at most one relatively prime to $|G|$, and $S$ is a sequence over $G$ with $|S|\geq |W|+|G|-1\geq |G|+1$, the maximum multiplicity of $S$ at most $|W|$, and $\sigma(W)\equiv 0\mod |G|$, then there exists a nontrivial subgroup $H$ such that every element  $g\in H$ can be represented as a weighted subsequence sum of the form $g=\Sum{i=1}{n}w_is_i$, with $s_1\cdots s_n$ a subsequence of $S$. We give two examples showing this does not hold in general, and characterize the counterexamples for large $|W|\geq \frac{1}{2}|G|$.

A theorem of Gao, generalizing an older result of Olson, says that if $G$ is a finite abelian group, and $S$ is a sequence over $G$ with $|S|\geq |G|+\D(G)-1$, then either every element of $G$ can be represented as a $|G|$-term subsequence sum from $S$, or there exists a coset $g+H$ such that all but at most $|G/H|-2$ terms of $S$ are from $g+H$. We establish some very special cases in a weighted analog of this theorem conjectured by Ordaz and Quiroz, and some partial conclusions in the remaining cases, which imply a recent result of Ordaz and Quiroz. This is done, in part, by extending a weighted setpartition theorem of Grynkiewicz, which we then use to also improve the previously mentioned result of Gao by showing that the hypothesis $|S|\geq |G|+\D(G)-1$ can be relaxed to $|S|\geq |G|+\dd^*(G)$, where $\dd^*(G)=\Sum{i=1}{r}(n_i-1)$. We also use this method to derive a variation on Hamidoune's conjecture valid when at least $\dd^*(G)$ of the $w_i$ are relatively prime to $|G|$.
\end{abstract}

\maketitle
\section{Notation}

We follow the conventions of \cite{GaoGe1} for notation concerning sequences over an abelian group. For real numbers $a,\, b \in \mathbb R$, we
set $[a, b] = \{ x \in \mathbb Z \mid a \le x \le b\}$.
Throughout, all abelian groups will be written additively.
Let $G$  be an abelian group, and let  $A,\, B \subseteq G$ be nonempty subsets. Then  $$A+B = \{a+b
\mid a \in A,\, b \in B \}$$  denotes their  \emph{sumset}. The  \emph{stabilizer}  of $A$ is defined as $H(A) = \{ g \in G \mid g +A = A\}$, and $A$ is called \emph{periodic}  if $H(A) \ne \{0\}$, and \emph{aperiodic} otherwise. If $A$ is a union of $H$-cosets (i.e., $H\leq H(A)$), then we say $A$ is \emph{$H$-periodic}. The order of an element $g\in G$ is denoted $\ord(g)$, and we use $\phi_H:G\rightarrow G/H$ to denote the natural homomorphism. We use $\gcd(a,b)$ to denote the greatest common divisor of $a,\,b\in\Z$.

\medskip

Let $\mathcal F(G)$ be the free monoid with basis $G$. The elements
of $\mathcal F(G)$ are called \emph{sequences} over $G$. We write
sequences $S \in \mathcal F (G)$ in the form
$$
S =  s_1\cdots s_r=\prod_{g \in G} g^{\vp_g (S)}\,, \quad \text{where} \quad
\vp_g (S)\geq 0\mbox{ and } s_i\in G.
$$
We call $|S|:=r=\Summ{g\in G}\vp_g(S)$ the \emph{length} of $S$, and  $\mathsf v_g (S)$  the \ {\it multiplicity} \ of $g$ in
$S$. The \emph{support} of $S$ is $$\supp(S):=\{g\in G\mid \vp_g(S)>0\}.$$  A sequence $S_1 $ is called a \emph{subsequence}  of
$S$ if $S_1 | S$  in $\mathcal F (G)$  (equivalently,
$\vp_g (S_1) \leq\vp_g (S)$  for all $g \in G$), and in such case, $S{S_1}^{-1}$ denotes the subsequence of $S$ obtained by removing all terms from $S_1$.
The \emph{sum} of $S$ is $$\sigma(S):=\Sum{i=1}{r}s_i=\Summ{g\in G}\vp_g(S)g,$$ and we use $$\h(S):=\max \{\vp_g (S) \mid g \in G \}$$ to denote the maximum multiplicity of a term of $S$. A sequence $S$ is $\emph{zero-sum}$ if $\sigma(S)=0$. Given any map $\varphi: G\rightarrow G'$, we extend $\varphi$ to a map of sequences,  $\varphi: \Fc(G)\rightarrow \Fc(G')$,  by letting $\varphi(S):=\varphi(s_1)\cdots \varphi(s_r)$.

\medskip

Next we introduce notation for weighted subsequence sums. Let $S\in\Fc(G)$ and $W\in \Fc(\Z)$ with $S=s_1\cdots s_r$, $W=w_1\cdots w_t$, and $s=\min\{r,\,t\}$ (more generally, we can let $W\in \Fc(G)$ if we endow $G$ with ring structure).
Define $$W\cdot S=\{\Sum{i=1}{s}w_{\sigma(i)}s_{\tau(i)} : \sigma\mbox{ a permutaion of } [1,t]\mbox{ and } \tau \mbox{ a permutation of } [1,r] \},$$ and for $1\leq n\leq s$, let
\ber\Sigma_n(W,S) &=& \left\{W'\cdot S': S'|S,\,W'|W\mbox{ and }|W'|=|S'|=n\right\}\nn\\\nn
\Sigma_{\leq
n}(W,S)&=&\bigcup_{i=1}^{n}\Sigma_i(W,S)\quad \mbox{ and }\quad \Sigma_{\geq n}(W,S)=\bigcup_{i=n}^{s}\Sigma_i(W,S),\\\nn\Sigma(W,S)&=&\Sigma_{\leq s}(W,S).\nn\eer If $W=1^{|S|}$, then $\Sigma(W,S)$ (and other such notation) is abbreviated by $\Sigma(S)$, which is the usual notation for the set of \emph{subsequence sums}. Note that $\Sigma_{|W|}(W,S)=W\cdot S$ when $|W|\leq |S|$.

\medskip

Finally, an \emph{$n$-setpartition} of a
sequence $S\in \Fc(G)$ is a factorization of $S=A_1\cdots A_n$ with $\h(A_i)=1$ for all $i$. By associating $A_i$ and $\supp(A_i)$, we consider each $A_i$ to also be a nonempty subset (in view of $\h(A_i)=1$), and use $A=A_1,\ldots,A_n$ to denote the $n$-setpartition $A$, so as to avoid confusion with $A_1\cdots A_n$, which equals the sequence $S$ partitioned/factorized by $A$. It is easily shown (see \cite{EGZ-II} \cite{rasheed} \cite{PhD-Dissertation}) that $S$ has an $n$-setpartition if and only if $\h(S)\leq n\leq |S|$, and if such is the case, then $S$ has an $n$-setpartition with sets of as near equal a size as possible (i.e., $||A_i|-|A_j||\leq 1$ for all $i$ and $j$). 

%
\section{Introduction}

Let $$G\cong C_{n_1}\oplus\ldots\oplus C_{n_r}$$ be a finite abelian group with $n_1|n_2|\ldots|n_r$, where $C_{n_j}$ denotes a cyclic group of order $n_j\geq 2$. Thus $r$ is the rank $\rr(G)$, $n_1\cdots n_r$ is the order $|G|$, and $n_r$ is the exponent $\exp(G)$. In 1961, Erd\H{o}s, Ginzburg and Ziv proved that every sequence $S\in \Fc(G)$ with $|S|\geq 2|G|-1$ has $0\in \Sigma_{|G|}(S)$ \cite{erdos} \cite{natbook}. This sparked the field of zero-sum problems, which has now seen much development and become an essential component in Factorization Theory (see  \cite{GaoGe1} \cite{Alfred-book} for a recent survey and text on the subject).

One of the oldest and most important invariants in this area is the \emph{Davenport Constant} of $G$, denoted $\D(G)$, which is the least integer so that $S\in \Fc(G)$ with $|S|\geq \D(G)$ implies $0\in\Sigma(S)$. A very basic argument shows \be\label{davconstant-triv-bounds}\dd^*(G)+1\leq \D(G)\leq |G|\ee (see \cite{Alfred-book}), where $$\dd^*(G):=\Sum{i=1}{r}(n_i-1).$$ Originally, the lower bound was favored as the likely truth, but later examples with $\D(G)>\dd^*(G)+1$ were found (see \cite{gao-ger-D(G)-counter} \cite{ger-schneider-D(G)-counter}), and it is not now well understood when $\dd^*(G)+1=\D(G)$ fails, though it is still thought that equality should hold for many instances (and known to be the case for a few) \cite{Alfred-book}.

Gao later linked the study of zero-sums with the study of $|G|$-term zero-sums (and hence results like the Erd\H{o}s-Ginzburg-Ziv Theorem), by showing that $\ell(G)=|G|+\D(G)-1$, where $\ell(G)$ is the least integer so that $S\in \Fc(G)$ with $|S|\geq \ell(G)$ implies $0\in\Sigma_{|G|}(S)$ \cite{Gao-M+D-1}. In the same paper, he also proved the following generalization of an older result of Olson \cite{olson1}.

\smallskip

\begin{theirtheorem}\label{thm-gao-coset-condition}Let $G$ be a finite abelian group, and let $S\in \Fc(G)$ with $|S|\geq |G|+\D(G)-1$. Then either $\Sigma_{|G|}(S)=G$ or there exists a coset $g+H$ such that all but at most $|G/H|-2$ terms of $S$ are from $g+H$.
\end{theirtheorem}

\smallskip

\noindent \noindent Thus the number $\ell(G)=|G|+\D(G)-1$ also guarantees that \emph{every} element (not just zero) can be represented as an $|G|$-term subsequence sum, provided no coset contains too many of the terms of $S$.

\medskip

In this paper, we concern ourselves with weighted zero-sum problems related to the above results, though some of our results are new in the non-weighted case as well. Such  variations were initiated by Caro in \cite{carosurvery} where he conjectured the following weighted version of the Erd\H{o}s-Ginzburg-Ziv Theorem, which, after much partial work \cite{alonconjprime} \cite{gao-wegz-partialcase} \cite{hamweightegzgroup} \cite{hamweightsrelprime}, was recently proven in \cite{WEGZ}. (Note the condition $\sigma(W)\equiv 0\mod \exp(G)$ is necessary, else $S$ with $\supp(S)=\{1\}$ would give a counterexample.)

\smallskip

\begin{theirtheorem}\label{thm-WEGZ} Let $G$ be a finite abelian group, and let $S\in \Fc(G)$ and $W\in \Fc(\Z)$ with $\sigma(W)\equiv 0\mod \exp(G)$. If $|S|\geq |W|+|G|-1$, then $0\in \Sigma_{|W|}(W,S)$.
\end{theirtheorem}

\smallskip

Since then, there have been several other results along these lines (see  \cite{adhikari1} \cite{adhikari2} \cite{griff} \cite{ordaz-quiroz} for some examples). However, the following conjecture of Hamidoune remained open \cite{hamweightegzgroup}.

\smallskip

\begin{conjecture}\label{conj-hamidoune}
Let $G$ be a finite, nontrivial abelian group, and let $S\in \F(G)$ and $W\in \Fc(\Z)$ with $|S|\geq |W|+|G|-1\geq |G|+1$ and $\sigma(W)\equiv 0\mod |G|$. If  $\h(S) \leq |W|$ and, for some $w_n\in \Z$,
$\gcd(w_i,|G|)=1$ for all $w_i|w_n^{-1}W$, then $\Sigma_{|W|}(W,S)$ contains a nontrivial subgroup.
\end{conjecture}

\smallskip

Hamidoune verified his conjecture in the case $|W|=|G|$ \cite{hamweightegzgroup}, and under the additional hypothesis of either  $\h(S)<|W|$ or $|W|\geq |G|$ or
$\gcd(w_i,|G|)=1$, for all $w_i|W$, Conjecture \ref{conj-hamidoune} follows from the result in \cite{WEGZ}. In Section 3, we give two examples which show that Conjecture \ref{conj-hamidoune} is false in general, and prove the following theorem, which characterizes the (rather limited) counter-examples for large $|W|\geq \frac{1}{2}|G|$.

\smallskip

\begin{theorem}\label{thm-conj-Hamidoune}
Let $G$ be a finite, nontrivial abelian group, and let $S\in \F(G)$ and $W\in \Fc(\Z)$ with $|S|\geq |W|+|G|-1\geq |G|+1$ and $\sigma(W)\equiv 0\mod |G|$. Suppose $\h(S) \leq |W|$ and, for some $w_n\in \Z$,
$\gcd(w_i,|G|)=1$ for all $w_i|w_n^{-1}W$.
If also $|W|\geq
\frac{1}{2}|G|$, then either:

(i) $\Sigma_{|W|}(W,S)$ contains a nontrivial subgroup, or

(ii) $|\supp(S)|=2$, $|W|=|G|-1$,
$G\cong \Z/2^r\Z$
 and $W\equiv x^{(n-1)/2}(-x)^{(n-1)/2}0\mod |G|$, for some $r,\,n,\,x\in \Z^+$.
\end{theorem}

\smallskip

Another open conjecture is the following weighted generalization of Theorem \ref{thm-gao-coset-condition} \cite{ordaz-quiroz}. We remark that in the same paper, they showed Conjecture \ref{conjetura} to be true when $|S|=2|G|-1$, and thus for cyclic groups.

\smallskip

\begin{conjecture} \label{conjetura}
Let $G$ be a finite abelian group, and let $W\in \Fc(\Z)$ with $\gcd(w_i,|G|)=1$ for all $w_i|W$,  $|W|=|G|$, and $\sigma(W) \equiv 0 \mod |G|$. If $S\in \Fc(G)$ with  $|S|
=|G| +\D(G)-1$, then either (i)  $\Sigma_{|G|}(W,S)=G$ or (ii) there exists a coset $g+H$ such that all but at most $|G/H|-2$ terms of $S$ are from $g+H$.
\end{conjecture}

\smallskip

In section 5, we prove some limited results related to Conjecture \ref{conjetura}. In particular, we verify it in the extremal case $\h(S)\geq \D(G)-1$ (allowing also $|S|\geq |G| +\D(G)-1$ provided $\h(S)\leq |G|$), and give a corollary that extends the result of \cite{ordaz-quiroz} and shows, when $\h(S)\leq \D(G)-1$, that the hypotheses of Conjecture \ref{conjetura} (assuming (ii) fails) instead imply $\Sigma_{|S|-|G|}(W,S)=G$. This latter result will follow from the following pair of theorems, which improve (for non-cyclic groups) a corollary from the end of \cite{WEGZ} (see also \cite{ccd} for the non-weighted version, of which this is also an improvement).

\smallskip

\begin{theorem}\label{thm-ccd-cor-davstar} Let $G$ be an abelian group of order $m$, let $W=w_1\cdots w_n$ be a sequence of integers relatively prime to $\exp(G)$, let $S\in \Fc(G)$ and let $S'|S$. Suppose $n\geq \dd^*(G)$ and $\h(S')\leq n\leq |S'|$. Then there exists some $S''|S$ with $|S''|=|S'|$ such that either:

(i) there exists an $n$-setpartition $A=A_1,\ldots,A_n$ of $S''$ such that $$|\Sum{i=1}{n}w_i\cdot A_i|\geq \min\{m,\,|S'|-n+1\},$$ or

(ii) there exists an $n$-setpartition $A=A_1,\ldots,A_n$ of $S''$, a proper, nontrivial subgroup $H\leq G$, and $g\in G$, such that:

(a) $(g+H)\cap A_i\neq \emptyset$ for all $i$, and $\supp(S{S''}^{-1})\subseteq g+H$,

(b) $A_i\subseteq g+H$ for $i\leq \dd^*(H)$ and $i> \dd^*(H)+\dd^*(G/H)$,

(c) $|\Sum{i=1}{n}w_i\cdot A_i|\geq (e+1)|H|$ and all but $e\leq |G/H|-2$ terms of $S$ are from $g+H$, and (d)

$$\Sum{i=1}{\dd^*(H)}w_i\cdot A_i=\left(\Sum{i=1}{\dd^*(H)}w_i\right)g+H.$$
\end{theorem}

\smallskip

\begin{theorem}\label{thm-ccd-bonus}Let $G$ be an abelian group of order $m$, let $W=w_1\cdots w_n$ be a sequence of integers relatively prime to $\exp(G)$, let $S\in \Fc(G)$ and let $S'|S$. Suppose $n\geq \dd^*(G)$ and $\h(S')\leq n\leq |S'|$.

Suppose there exists a nontrivial subgroup $K$, $g'\in G$, and a $\dd^*(K)$-setpartition $B=B_1,\ldots,B_{\dd^*(K)}$ of a subsequence $T|S$ with $T\in \Fc(g'+K)$, such that $$\Sum{i=1}{\dd^*(K)}w_i\cdot B_i=\left(\Sum{i=1}{\dd^*(K)}w_i\right)g'+K$$ and $T^{-1}S$ contains at least $n-\dd^*(K)+|S|-|S'|$ terms from $g'+K$, and let $K$  be a maximal such subgroup. Then the following hold.

(i) If $K=G$, then there is an $n$-setpartition $A=A_1,\ldots,A_n$ of a subsequence $S''|S$ such that $|S'|=|S''|$ and  $$\Sum{i=1}{n}w_i\cdot A_i=G.$$

(ii) If $K<G$, then the conclusion of Theorem \ref{thm-ccd-cor-davstar}(ii) holds with $H=K$.
\end{theorem}

\smallskip

Theorem \ref{thm-ccd-cor-davstar} allows the result to applied when $n\geq \dd^*(G)$, rather than $n\geq \frac{|G|}{p}-1$ (as in the original corollary), where $p$ is the smallest prime divisor of $|G|$ (note, for non-cyclic groups, the number $\dd^*(G)$ is generally much smaller than $\frac{|G|}{p}-1$), and contains similar improvements of bounds present in (ii)(b). However, the bound present in (ii)(c) remains unaltered, and improvements here would likely be more difficult. Theorem \ref{thm-ccd-bonus} will be used to prove Theorem \ref{thm-ccd-cor-davstar}, and also gives a way to force Theorem \ref{thm-ccd-cor-davstar}(ii) to hold.

As a second consequence of Theorems \ref{thm-ccd-cor-davstar} and \ref{thm-ccd-bonus}, we prove the following variation on Theorem \ref{thm-conj-Hamidoune}, which extends Hamidoune's result from \cite{hamweightsrelprime} by showing that it is only necessary to have at least $\dd^*(G)$ of the weights relatively prime to $\exp(G)$.

\smallskip

\begin{corollary}\label{thm-conj-Hamidoune-var}
Let $G$ be a finite, nontrivial abelian group, and let $S\in \F(G)$ and $W\in \Fc(\Z)$ with $|S|\geq |W|+|G|-1$ and $\sigma(W)\equiv 0\mod \exp(G)$. Suppose $\h(S) \leq |W|$ and, for some $W'|W$ with $|W'|=t$, $\gcd(w_i,\exp(G))=1$ for all $w_i|{W'}^{-1}W$. If also $|W|\geq
\dd^*(G)+t$, then $\Sigma_{|W|}(W,S)$ contains a nontrivial subgroup.
\end{corollary}

\smallskip

As a third consequence, we improve Theorem \ref{thm-gao-coset-condition} by relaxing the required hypothesis from $|S|\geq |G|+\D(G)-1$ to $|S|\geq |G|+\dd^*(G)$ (recall from (\ref{davconstant-triv-bounds}) that $\D(G)-1\geq \dd^*(G)$). This should be put in contrast to the fact that $\ell(G)=|G|+\D(G)-1>|G|+\dd^*(G)$ is in general possible (since $\D(G)-1>\dd^*(G)$ is possible). The methods of employing Theorems \ref{thm-ccd-cor-davstar} and \ref{thm-ccd-bonus} from these three applications should also be applicable for other zero-sum problems.

\section{On Conjecture \ref{conj-hamidoune}}

We begin by giving the two counter examples to Conjecture \ref{conj-hamidoune}.

\medskip

\textbf{Example 1. } Let $p\equiv -1\mod 4$ be a prime, let $G=\Z/p\Z$, let
$n=\frac{p-1}{2}$, let $W=1^{(n-1)/2}(-1)^{(n-1)/2}0$, and let
$S=0^n1^n2^n$. Note that $\h(S)=n=|W|$, that $|S|=3n=|W|+|G|-1$, that $\sigma(W)=0$, and that
$$\Sigma_{|W|}(W,S)=\Sum{i=1}{(n-1)/2}\{0,1,2\}-\Sum{i=1}{(n-1)/2}\{0,1,2\}=G\setminus\{\frac{p+1}{2},\,\frac{p-1}{2}\}.$$
Thus $G\nsubseteq \Sigma_{|W|}(W,S)$, which, since $|G|$ is prime, implies
$\Sigma_{|W|}(W,S)$ does not contain a nontrivial subgroup.

\medskip

\textbf{Example 2. } Let $m=2^r$, let $G=\Z/m\Z$, let $n=m-1$, let
$W=1^{(n-1)/2}(-1)^{(n-1)/2}0$, and let $S=0^n1^n$. Note that $\h(S)=n=|W|$, that
$|S|=2n=|W|+|G|-1$, that $\sigma(W)=0$, and that
$$\Sigma_{|W|}(W,S)=\Sum{i=1}{(n-1)/2}\{0,1\}-\Sum{i=1}{(n-1)/2}\{0,1\}=G\setminus\{\frac{m}{2}\}.$$
Hence, since every nontrivial subgroup of $G\cong \Z/2^r\Z$
contains the unique element of order 2, namely $\frac{m}{2}=2^{r-1}$, it follows that $\Sigma_{|W|}(W,S)$ does not
contain a nontrivial subgroup.

\medskip

For the proof of Theorem \ref{thm-conj-Hamidoune}, we will need to make use of the Kemperman critical pair theory (though an isoperimetric approach would also be viable, see e.g. \cite{plagne}). We begin by stating Kneser's Theorem  \cite{kt-asymptotic} \cite{kt} \cite{natbook} \cite{taobook}.

\smallskip

\begin{theirtheorem}[Kneser's Theorem] Let $G$ be an abelian group, and let $A_1,\ldots,A_n\subseteq G$ be finite, nonempty subsets. Then $$|\Sum{i=1}{n}\phi_H(A_i)|\geq \Sum{i=1}{n}|\phi_H(A_i)|-n+1,$$ where $H=H(\Sum{i=1}{n}A_i)$.
\end{theirtheorem}

\smallskip

Next we continue with the following two simple cases of Kemperman's Structure Theorem \cite{kst}. The reader is directed to \cite{quasi-periodic} \cite{PhD-Dissertation} \cite{KST+1}  \cite{lev-kemp} for more detailed exposition regarding Kemperman's critical pair theory, including the (somewhat lengthy and involved) statement of the Kemperman Structure Theorem. In what follows, a set $A\subseteq G$ is \emph{quasi-periodic} if there is a nontrivial subgroup $H$ (the \emph{quasi-period}) such that $A=A_0\cup A_1$ with $A_0$ nonempty and $H$-periodic and $A_1$ a subset of an $H$-coset.

\smallskip

\begin{lemma}\label{lemaKST}
Let $A_1, \ldots , A_n,$ be a collection of $n\geq 3$ finite
subsets in an abelian group $G$ of order $m$  with $0 \in A_i$
and $|A_i| \geq 2$ for all $i$. Moreover, suppose each $A_i$ is not quasi-periodic and
$ \langle A_i\rangle =G$. If $\Sum{i=1}{n} A_i$ is aperiodic
and
 \be\label{kemplem-hypoo}|\Sum{i=1}{n} A_i| = \Sum{i=1}{n} |A_i|-n + 1,\ee then the $A_i$ are arithmetic progressions with common difference.
\end{lemma}

\begin{proof} We provide a short proof using the formulation (including relevant notation and definitions) of Kemperman's Structure Theorem as given in \cite{kst}.

Since $\Sum{i=1}{n} A_i$ is aperiodic, it follows that $A_j+A_k$ is aperiodic for any $j\neq k$. Thus Kneser's Theorem implies $|A_j+A_k| \geq |A_j|+|A_k| -1$, and we
must have $$|A_j+A_k| = |A_j|+|A_k|-1,$$ else Kneser's Theorem would imply $$|\Sum{i=1}{n}A_i|\geq \underset{i\neq j,\,k}{\Sum{i=1}{n}}|A_i|+|A_j+A_k|-(n-1)+1\geq \Sum{i=1}{n}|A_i|-n+2,$$ contradicting (\ref{kemplem-hypoo}). Thus we can apply Kemperman's Structure Theorem to an arbitrary pair $A_j$ and $A_k$ with $j\neq k$.

Since
 $A_i$ is not quasi-periodic, for $i=j,\,k$, we conclude from the Kemperman Structure Theorem that $(A_j,A_k)$ is an elementary pair of type (I), (II), (III) or (IV). Since $|A_j|,\,|A_k|\geq 2$ and $A_j+A_k$ is aperiodic, we cannot have type (I) or (III). Since $n\geq 3$, since $|A_i|\geq 2$ for all $i$, and since $\Sum{i=1}{n}A_i$ is aperiodic (and in particular, $|\Sum{i=1}{n}A_i|<|G|)$, it follows in view of Kneser's Theorem that $|A_j+A_k|<|\Sum{i=1}{n}A_i|<|G|$. Thus, in view of $0 \in A_j$ and
$ \langle A_j\rangle =G$, it follows that we cannot have type (IV) and that $|A_j|,\,|A_k|\leq |G|-2$. Hence $(A_j,A_k)$ is an elementary pair of type (II), i.e., $A_j$ and $A_k$ are arithmetic progressions of common difference (say) $d\in G$. Note $\ord(d)=|G|$, since $\langle A_j\rangle =G$. Since the difference $d$ of an arithmetic progression $A$ is unique up to sign when $2\leq |A|\leq \ord(d)-2$, since $2\leq |A_j|,\,|A_k|\leq |G|-2$, and since $A_j$ and $A_k$ with $j\neq k$ were arbitrary, it now follows that all the $A_i$ are arithmetic progressions of common difference $d$, as desired.
\end{proof}

\smallskip

\begin{lemma}\label{lemaKST-triv-case}
Let $G$ be an abelian group and let $A,\,B\subseteq G$ be finite with $|A|\geq 2$ and $|B|=2$. If neither $A$ nor $B$ is quasi-periodic and $|A+B|=|A|+|B|-1$, then $A$ and $B$ are arithmetic progressions of common difference.\end{lemma}

\begin{proof} This follows immediately from the Kemperman Structure Theorem or may be taken as an easily verified observation.
\end{proof}

\smallskip

The following result from \cite{WEGZ} will also be used.

\smallskip

\begin{theirtheorem}\label{thm-wegzi} Let $G$ be a finite, nontrivial abelian group, let $W=w_1\cdots w_n$ be a sequence of integers with $\sigma(W)\equiv 0\mod \exp(G)$, and let $S\in \Fc(G)$ with $|S|\geq |W|+|G|-1$. Suppose $S$ has an $n$-setpartition $A=A_1,\ldots,A_n$
such that $|w_i\cdot A_i|=|A_i|$ for all $i$. Then there exists a
nontrivial subgroup $H$ of $G$ and an $n$-setpartition
$A'=A'_1,\ldots,A'_n$ of $S$ with $$H\subseteq \Sum{i=1}{n}w_i\cdot
A'_i\subseteq \Sigma_{|W|}(W,S)$$ and $|w_i\cdot A'_i|=|A'_i|$ for all $i$.
\end{theirtheorem}

\smallskip

We now proceed with the proof of Theorem \ref{thm-conj-Hamidoune}.

\smallskip

\begin{proof} Let $m=|G|$. By considering $G$ as a $\Z/m\Z$-module (for notational convenience), we may w.l.o.g. consider $W$ as a sequence from $\Z/m\Z$, say w.l.o.g. $W=w_1\cdots w_n$, where $\ord(w_i)=m$ for $i\leq n-1$ (in view of the hypothesis $\gcd(w_i,|G|)=1$ for $i\leq n-1$). Observe that we may assume $|S|=n+m-1$ (since $n\geq 2$, so that, if $|\supp(S)|\geq 3$, then we can remove terms from $S$ until there are only $n+m-1\geq 3$ left while preserving that $|\supp(S)|\geq 3$), and that there are distinct $x,\,y\in G$ with
$x^ny^n|S$ such that $w_n(x-y)=0$, else Theorem \ref{thm-wegzi} implies the theorem (as if such is not the case, then there would exist, in view of $\h(S)\leq |W|=n$, an $n$-setpartition of $S$ satisfying the hypothesis of Theorem \ref{thm-wegzi}). Since $\sigma(W)=0$, we may w.l.o.g. by translation
assume $x=0$. If $\ord(y)<m$, then, since $w_iy\in \langle y\rangle $ and $$n-1\geq
\frac{m}{2}-1\geq \ord(y)-1=\ord(w_iy)-1,$$ for $i\leq n-1$ (in view of $\ord(w_i)=m$ for $i\leq n-1$), it would
follow in view of Kneser's Theorem that
$$\langle y\rangle=\Sum{i=1}{n-1}\{0,w_iy\}=\Sum{i=1}{n-1}w_i\cdot \{0,\,y\}+w_n\cdot
0\subseteq \Sigma_{|W|}(W,S),$$ as desired. Therefore we may assume
$\ord(y)=m$, whence w.l.o.g. $G$ is cyclic and $y=1$.
Consequently, since $\sigma(W)=0$, $w_n(x-y)=0$ and $x=0$, it follows that $w_n=0$ and
$\sigma(W')=0$, where $W':=Ww_n^{-1}$.

Since $n\geq \frac{m}{2}$, $0^n1^n|S$ and $|S|=n+m-1$, it follows
that \be\label{daffy}2n\leq |S|\leq 3n-1.\ee Hence let $A=A_1,\ldots,A_{n-1}$ be an
arbitrary $(n-1)$-setpartition of $S':=S(01)^{-1}$. Note
$\{0,\,1\}\subseteq A_i$ for all $i$, so that \be\label{bobo}0\in
\Sum{i=1}{n-1}w_i\cdot A_i+w_n\cdot 0\subseteq \Sigma_{|W|}(W,S).\ee Thus we may
assume $\Sum{i=1}{n-1}w_i\cdot A_i$ is aperiodic, else the proof is
complete. Consequently, Kneser's Theorem and $\ord(w_i)=m$ for $i\leq n-1$ imply $|\Sum{i=1}{n-1}w_i\cdot
A_i|\geq \Sum{i=1}{n-1}|A_i|-(n-1)+1=m-1$, whence
\be\label{walkthewalk}|\Sum{i=1}{n-1}w_i\cdot A_i|=
\Sum{i=1}{n-1}|A_i|-(n-1)+1=m-1,\ee else $G\subseteq
\Sum{i=1}{n-1}w_i\cdot A_i+w_n\cdot 0\subseteq\Sigma_{|W|}(W,S)$, as
desired.

\medskip

Suppose $m$ is not a prime power. Then we can choose $H,\,K\leq G$
with $|H|$ and $|K|$ distinct primes, so that $H\cap K=\{0\}$. In view
of (\ref{walkthewalk}), it follows that $\Sigma_{|W|}(W,S)$ is missing
exactly one element, which in view of (\ref{bobo}) cannot be zero.
Consequently, either $H\subseteq \Sigma_{|W|}(W,S)$ or $K\subseteq \Sigma_{|W|}(W,S)$, as desired. So we may assume $m=p^r$ for some prime $p$ and
$r\geq 1$.

\medskip

\textbf{Claim A: } If
$x(-x)|W'$, for some $x\in \Z/m\Z$, then  $|S|=2n$ or $|S|=3n-1$, else the proof is complete.

\medskip

\begin{proof}

Suppose the claim is false. Thus (\ref{daffy}) implies $2n+1\leq |S|\leq 3n-2$, so that $n\geq 3$, and it follows by the pigeonhole principle that
$|A_i|\leq 2$ for some $i$, say $i=n-1$, and that $|A_j|\geq 3$
for some $j$, say $j=n-2$, whence we may w.l.o.g. assume $x=w_{n-2}$ and
$-x=w_{n-1}$. Let $g\in A_{n-2}\setminus \{0,\,1\}$ (in view of $|A_{n-2}|=|A_j|\geq 3$). Observe
that \begin{multline}\label{longlike}\Sum{i=1}{n-2}w_i\cdot
A_i+w_{n-1}\cdot(A_{n-1}\cup \{g\})= \left(\Sum{i=1}{n-1}w_i\cdot
A_i\right)\bigcup\\ \left( \Sum{i=1}{n-3}w_i\cdot A_i+w_{n-2}\cdot
(A_{n-2}\setminus \{g
\})+w_{n-1}\cdot (A_{n-1}\cup\{g\})\right)\bigcup\left(
\Sum{i=1}{n-3}w_i\cdot A_i+w_{n-2}g+w_{n-1}g\right).\end{multline} Note that the first two terms of the right
hand side of (\ref{longlike}) are contained in $\Sigma_{|W|}(W,S)$.
Moreover,
$$w_{n-2}g+w_{n-1}g=xg+(-x)g=0=w_{n-2}\cdot 0+w_{n-1}\cdot 0\in w_{n-2}\cdot
A_{n-2}+w_{n-1}\cdot A_{n-1},$$ so that the third term of the
right hand side of (\ref{longlike}) is contained in
$\Sum{i=1}{n-1}w_i\cdot A_i+w_n\cdot 0\subseteq \Sigma_{|W|}(W,S)$ as well. Consequently,
it follows from (\ref{longlike}) that
\be\label{yarrh}\Sum{i=1}{n-2}w_i\cdot
A_i+w_{n-1}\cdot(A_{n-1}\cup \{g\})\subseteq \Sigma_{|W|}(W,S).\ee
However, since $\Sum{i=1}{n-1}w_i\cdot A_i$ is aperiodic and
$w_{n-1}g\notin w_{n-1}\cdot A_{n-1}$ (in view of $\ord(w_{n-1})=m$, $|A_{n-1}|=|A_i|=2$, and $\{0,1\}\subseteq A_i$), it follows from
Kneser's theorem that $$|\Sum{i=1}{n-2}w_i\cdot
A_i+w_{n-1}\cdot(A_{n-1}\cup \{g\})|> \Sum{i=1}{n-1}|
A_i|-(n-1)+1=m-1.$$ Thus (\ref{yarrh}) implies
that $G\subseteq \Sigma_{|W|}(W,S)$, as desired, completing the proof of Claim A.\end{proof}

\medskip

If $n=2$, then $\sigma(W')=0$ implies $w_1=0$,
contradicting $\ord(w_i)=m$ for $i\leq n-1$. Therefore we may assume
$n\geq 3$.

\medskip

Suppose $|S|=2n$ (so that $S=0^n1^n$). Since $|S|=n+m-1=2n$ and $n\geq 3$, it follows
that $m=n+1\geq 4$. In view of (\ref{walkthewalk}) and Lemmas \ref{lemaKST} and \ref{lemaKST-triv-case}, it
follows that each $w_i\cdot A_i=w_i\cdot\{0,1\}=\{0,w_i\}$ is an arithmetic
progression with common difference. Consequently,
it follows that $w_i=\pm w_j$ for all $i,\,j\leq n-1$. Since $n-1< m$, since
$\sigma(W')=0$, and since $\ord(w_i)=m$ for all $i\leq n-1$, it
follows that the $w_i$ cannot all be equal. As a result, $w_i=\pm
w_j$ for all $i,\,j\leq n-1$ implies that $w_i=\pm x$ for all $i\leq n-1$,
with $(x)(-x)|W'$ (for some $x\in \Z/m\Z$), whence $\sigma(W')=0$ further implies that
$W'=x^{(n-1)/2}(-x)^{(n-1)/2}$ with $n-1$ even. Hence, since
$m=n+1$ is a prime power, it follows that $m=2^r$, and we see that (ii) holds. So we may assume $|S|>2n$.

\medskip

Suppose $n=3$. Then $\sigma(W')=0$ implies that
$w_1=-w_2$. Thus, since $2n<|S|$ and $x(-x)|W'$, where $x=w_1$, it follows in view of
Claim A that $|S|=3n-1=8$. Since $|S|=n+m-1=m+2$, this
implies that $m=6$, contradicting that $m$ is a prime power. So we
may assume $n\geq 4$.

\medskip

In view of (\ref{daffy}), choose $A$ such that $|A_i|\in \{2,\,3\}$ for all $i$ (possible by the remarks from Section 1). If, for some $j$, there is
$g\in A_j\setminus \{0,1\}$ such that $\{x,g\}$ is a
coset of a cardinality two subgroup $H$, where $x\in \{0,1\}$,
then
$$\underset{i\neq j}{\Sum{i=1}{n-1}}w_ix+w_j\cdot
\{x,g\}+w_n\cdot 0$$ is an $H$-periodic
subset of $\Sigma_{|W|}(W,S)$ that contains $\Sum{i=1}{n-1}w_ix=\sigma(W')\cdot x=0$; thus $H\subseteq \Sigma_{|W|}(W,S)$, as desired. Therefore
we may assume otherwise, and consequently that no $A_j$ is quasi-periodic (in view of $|A_j|\leq 3$ and $\Sum{i=1}{n-1}A_i$ aperiodic).

As a result, it follows, in view of $n\geq 4$, (\ref{walkthewalk}) and Lemma \ref{lemaKST}, that the
$w_i\cdot A_i$ are all arithmetic progressions of common
difference. Thus each $A_i$ is an arithmetic progression
of length two or three that contains $\{0,1\}$. Hence, since
$n\geq 4$ and $n+m-1=|S|> 2n$, so that $$m\geq 5,$$ it follows
that each  $A_i$ is an arithmetic progression with difference $1$
or $\frac{m+1}{2}$ (which both are of order $m$), and thus each $w_i\cdot A_i$ is an arithmetic
progression with difference $w_i$ or $w_i\cdot \frac{m+1}{2}$.

Thus, since $n-1\geq 3$, it follows by the pigeonhole principle that
there is a pair $A_{j}$ and $A_{k}$, with $j\neq k$, that are arithmetic
progressions with common difference $d$, where $\ord(d)=m$. Thus
$w_{j}\cdot A_{k}$ and $w_{k}\cdot A_{k}$ are arithmetic
progression with common difference $w_{j}d=\pm
w_{k}d$, implying $w_{j}=\pm w_{k}$ (since $\ord(d)=m$). Since the indexing for the $w_i$ was arbitrary, then, by applying this
argument to all possible permutations of the indices of the $w_i$ (leaving $w_n$ fixed), we conclude that $w_i=\pm
w_j$ for all $i,\,j\leq n-1$. As in the case $|S|=2n$, we cannot have
all the $w_i$, with $i\leq n-1$, equal to each other (in view of $\sigma(W')=0$ and $n-1<m=\ord(w_i)$), whence
$W=x^{(n-1)/2}(-x)^{(n-1)/2}0$ and $n$ is odd, for some $x\in \Z/m\Z$.

Thus, from claim $A$ and $|S|>2n$, we infer that $n+m-1=|S|=3n-1$, implying $2n=m$. Hence $m$ is
even. Thus, since $m$ is a prime power, it follows that $m=2^r$,
whence $2n=m=2^r\geq 5$ implies that $n$ is even, a contradiction,
completing the proof.
\end{proof}

\section{On $\dd^*(G)$}

The main goal of this section is to prove the following pair of seemingly innocuous lemmas, which will be needed for the proof of Theorem \ref{thm-ccd-cor-davstar}. Lemma \ref{d*lemma} should be compared with the similar \cite[Proposition 5.1.11]{Alfred-book}, whose proof is much easier.

\smallskip

\begin{lemma}\label{d*lemma}If $G$ is a finite abelian group and $H\leq G$, then $$\dd^*(H)+\dd^*(G/H)\leq \dd^*(G).$$
\end{lemma}

\smallskip

\begin{lemma}\label{lem-split} Let $G$ be a finite abelian group, let $A\subseteq G$ be finite with $|A|\geq 2$, let $H=\langle -a_0+A\rangle$, where $a_0\in A$, and let $W=w_1\cdots w_{\dd^*(H)}$ be a sequence of integers relatively prime to $\exp(H)$. Then $$\Sum{i=1}{\dd^*(H)}w_i\cdot A=\left(\Sum{i=1}{\dd^*(H)}w_i\right)a_0+H.$$
\end{lemma}

\smallskip

We first gather some basic results from algebra. Proposition \ref{lem-dual} is easily proved from the machinery of dual groups, and Proposition \ref{lem-align} from the notion and basic properties of independent elements.

\smallskip

\begin{proposition}\label{lem-dual} Let $G$ be a finite abelian group and $H\leq G$. Then there exists $K\leq G$ such that $K\cong G/H$ and $G/K\cong H$.
\end{proposition}

\begin{proof}
Since finite abelian groups are self-dual \cite[Theorem I.9.1]{lang}, this follows from \cite[Corollory I.9.3]{lang}.
\end{proof}

\smallskip

\begin{proposition}\label{lem-align} Let $G$ be a finite abelian group, say $G\cong \bigoplus_{i=1}^{r}C_{m_i}\cong \bigoplus_{i=1}^{l}\left(\bigoplus_{j=1}^{r}C_{p_i^{k_{i,j}}}\right)$, with $1<m_1|\ldots |m_r$, each $p_i$ a distinct prime, and $1\leq k_{i,1}\leq \ldots \leq k_{i,r}$. If $H\leq G$, then $$H\cong \bigoplus_{i=1}^{r}C_{m'_i}\cong \bigoplus_{i=1}^{l}\left(\bigoplus_{j=1}^{r}C_{p_i^{k'_{i,j}}}\right),$$ with $1\leq m'_1|\ldots |m'_r$ and $m'_i|m_i$ and $1\leq k'_{i,1}\leq \ldots\leq  k'_{i,r}$ and $k'_{i,j}\leq k_{i,j}$, for all $i$ and $j$. Moreover, if $m'_s=m_s$ for some $s$, then $k'_{i,s}=k_{i,s}$ for all $i$.
\end{proposition}

\begin{proof} Since $m_j=p_1^{k_{1,j}}p_2^{k_{2,j}}\cdots p_l^{k_{l,j}}$ (see \cite[Section II.2]{hungerford}), it suffices to show $k'_{i,j}\leq k_{i,j}$ for all $i$ and $j$. For this, it suffices to consider $p$-groups (the case $l=1$). We may assume $k'_{1,1}\leq \ldots \leq k'_{1,r}$, and now, if the the proposition is false, then $k'_{1,j}>k_{1,j}$ for some $j$, whence $H$, and hence also $G$, contains $r-j+1$ independent elements of order at least $p_1^{k_{1,j}+1}$, say $e_1,\ldots,e_{r-j+1}$. But now $p_1^{k_{1,j}}e_1,\ldots p_1^{k_{1,j}}e_{r-j+1}$ are $r-j+1$ independent elements in $p_1^{k_{1,j}}\cdot G$ (the image of $G$ under the multiplication by $p_1^{k_{1,j}}$ map), which has total rank $\rr^*(p_1^{k_{1,j}}\cdot G)$ at most $r-j$ (in view of $k_{i,1}\leq \ldots \leq k_{i,r}$), contradicting that the total rank of a group is the maximal number of independent elements (see \cite[Apendix A]{Alfred-book}).
\end{proof}

\smallskip

The next lemma will provide the key inductive mechanism for the proof of Lemma \ref{d*lemma}.

\smallskip

\begin{lemma}\label{lem-key} Let $G$ be a finite abelian group, say $G\cong \bigoplus_{i=1}^{r}C_{m_i}$, with $1<m_1|\ldots |m_r$, and let $H\leq G$, say $H\cong \bigoplus_{i=1}^{r}C_{m'_i}$, with $1\leq m'_1|\ldots |m'_r$. If $m'_t=m_t$ for some $t$, then there exists a subgroup $K\leq H$ such that $K\cong C_{m_t}$ and $K$ is a direct summand in both $H$ and $G$.
\end{lemma}

\begin{proof} Let $G\cong \bigoplus_{i=1}^{l}\left(\bigoplus_{j=1}^{r}C_{p_i^{k_{i,j}}}\right)$ and $H\cong \bigoplus_{i=1}^{l}\left(\bigoplus_{j=1}^{r}C_{p_i^{k'_{i,j}}}\right)$, with each $p_i$ a distinct prime, $1\leq k_{i,1}\leq \ldots \leq k_{i,r}$ and $1\leq k'_{i,1}\leq \ldots \leq k'_{i,r}$ for all $i$.
In view of Proposition \ref{lem-align} and our hypotheses, we have $m'_i|m_i$ and $k'_{i,j}\leq k_{i,j}$, for all $i$ and $j$, and $k'_{i,t}=k_{i,t}$ for all $i$. Thus it suffices to prove the lemma for $p$-groups, and so we assume $m_i=p^{s_i}$ and $m'_i=p^{s'_i}$ for some prime $p$.

By hypothesis, $H$ contains $r-t+1$ independent elements $f_1,\ldots,f_{r-t+1}$ of order at least $m_t=p^{s_t}$ (by an appropriate subselection of elements from a basis of $H$). Let $e_1,\ldots,e_r$ be a basis for $G$ with $\ord(e_i)=p^{s_i}$, and let  $f_j=\Sum{i=1}{r}\alpha_{j,i}e_i$, where $\alpha_{j,i}\in\Z$. If $$\ord(f_j)=\ord(\alpha_{j,i}e_i)=\ord(e_i)=p^{s_t},$$ for some $i$ and $j$, then $e_1,\ldots,e_{i-1},f_j,e_{i+1},\ldots,e_r$ is also a basis for $G$, and the result follows with $K=\langle f_j\rangle$. So we may assume otherwise.

Now $f'_1:=p^{s_t-1}f_1,f'_2:=p^{s_t-1}f_2,\ldots,f'_{r-t+1}:=p^{s_t-1}f_{r-t+1}$ are $r-t+1$ independent elements in $p^{s_t-1}\cdot G$. However, in view of the conclusion of the previous paragraph, each $p^{s_t-1}f_j$ with $\ord(f_j)=p^{s_t}$ must lie in the span of $p^{s_t-1}e_{t+1},\ldots,p^{s_t-1}e_r$ (as $\ord(e_i)\leq p^{s_t}$ for $i\leq t$).

Let $\phi_L:p^{s_t-1}\cdot G\rightarrow (p^{s_t-1}\cdot G)/L$, where $L=\langle p^{s_t-1}e_1,\ldots,p^{s_t-1}e_t\rangle$, be the natural homomorphism. Then $\phi_L(f'_1),\ldots,\phi_L(f'_{r-t+1})$ are $r-t+1$ independent elements in $\phi_L(p^{s_t-1}\cdot G)$, as the following argument shows. Take any relation $$0=\Sum{i=1}{r-t+1}\alpha_i\phi_L(f'_i)=\Summ{i\in I}\alpha_i\phi_L(f'_i)+\Summ{i\notin I} \alpha_i\phi_L(f'_i),$$ where $i\in I$ are those indices such that $\ord(f'_i)>p$ (and thus $\ord(f_i)>p^{s_t}$) and $\alpha_i\in \Z$. Then, in view of the conclusion of the previous paragraph, we see that $$0=\Sum{i=1}{r-t+1}p^{s'}\alpha_if'_i$$ is a relation in $p^{s_t-1}\cdot G$, where $s':=\max\{0,\,1-\min\{\vp_p(\alpha_i)\mid i\in I\}\}$ (here $\vp_p(\alpha_i)$ is the $p$-valuation of $\alpha_i\in \Z$).  If $s'=0$, then the independence of the $f'_i$ implies that $\alpha_if'_i=0$, and thus $\phi_L(\alpha_if'_i)=\alpha_i\phi_L(f'_i)=0$, for all $i$. If $s'=1$, then the definition of $s'$ implies that $\vp_p(\alpha_j)=0$ for some $j\in I$, whence $\ord(\alpha_jf'_j)>p$ follows from the definition of $I$. As a result, $p\alpha_jf'_j\neq 0$, contradicting that the $f'_i$ are independent. Thus $\phi_L(f'_1),\ldots,\phi_L(f'_{r-t+1})$ are $r-t+1$ independent elements in $\phi_L(p^{s_t-1}\cdot G)$, which is a group of total rank at most $r-t$, contradicting that the total rank is the maximal number of independent elements (see \cite[Apendix A]{Alfred-book}). This completes the proof.
\end{proof}

\smallskip

We can now prove Lemma \ref{d*lemma}.

\smallskip

\begin{proof} If $G$ is cyclic, then $\dd^*(G)=|G|-1$, $\dd^*(H)=|H|-1$ and $\dd^*(G/H)=\frac{|G|}{|H|}-1$. Hence $\dd^*(G)\geq \dd^*(H)+\dd^*(G/H)$ follows from the general inequality $xy\geq x+y-1$ for $x,\,y\in \Z_{\geq 1}$. Therefore we may assume $\rr(G)\geq 2$ and proceed by induction on the rank $\rr(G)=r$.

Let $G\cong \bigoplus_{i=1}^{r}C_{m_i}$, $H\cong \bigoplus_{i=1}^{r}C_{m'_i}$ and $G/H\cong \bigoplus_{i=1}^{r}C_{m''_i}$, with $1<m_1|\ldots |m_r$ and $1\leq m'_1|\ldots |m'_r$ and $1\leq m''_1|\ldots |m''_r$. In view of Propositions \ref{lem-align} and \ref{lem-dual}, we see that $m'_i|m_i$ and $m''_i|m_i$ for all $i$. Hence, if $m'_i<m_i$ and $m''_i<m_i$ for all $i$, then $m'_i\leq \frac{1}{2}m_i$ and $m''_i\leq \frac{1}{2}m_i$, whence $m'_i-1+m''_i-1<m_i-1$; consequently, summing over all $i$ yields the desired bound $\dd^*(G)\geq \dd^*(H)+\dd^*(G/H)$. Therefore we may assume $m'_s=m_s$ or $m''_s=m_s$ for some $s$, and in view of Proposition \ref{lem-dual}, we may w.l.o.g. assume $m'_s=m_s$.

Now applying Lemma \ref{lem-key}, we conclude that there are subgroups $K,\,H_0\leq H$ and $G_0\leq G$ such that $H=K\oplus H_0$ and $G=K\oplus G_0$ with $K\cong C_{m_s}$. Moreover, we can choose the complimentary summand $H_0$ such that $H_0\leq G_0$. Note $\dd^*(H)=\dd^*(K)+\dd^*(H_0)$ and $\dd^*(G)=\dd^*(K)+\dd^*(G_0)$, while $G/H=(K\oplus G_0)/(K\oplus H_0)\cong G_0/H_0$, so that $\dd^*(G_0/H_0)=\dd^*(G/H)$. Thus $\dd^*(G)\geq \dd^*(H)+\dd^*(G/H)$ follows by applying the induction hypothesis to $G_0$ with subgroup $H_0$.
\end{proof}

\smallskip

Having established Lemma \ref{d*lemma}, we conclude the section with the proof of Lemma \ref{lem-split}.

\smallskip

\begin{proof}
By translation, we may w.l.o.g. assume $a_0=0\in A$ and $H=G$. Since $|A|\geq 2$, we have $\langle A\rangle =H=G$ nontrivial.
Let $K\leq H=G$ be the maximal subgroup such that there exists a subset $B\subseteq A$ with $0\in B$, $K=\langle B\rangle$ and
\be\label{maxi-wo} |\Sum{i=1}{\dd^*(K)}w_i\cdot B|=|K|,\ee if such $K$ exists, and otherwise let $K=B=\{0\}$. We may assume $K<H=G$, else the proof is complete.

Since $\langle B\rangle =K\neq G$ and $\langle A\rangle =G$, choose $g\in A\setminus B$ such that $\langle B'\rangle :=K'>K$, where $B'=B\cup \{g\}$. Let $L=\langle g\rangle$. Note $K'/K=(K+L)/K\cong L/(K\cap L)$ is cyclic. Hence, in view of Lemma \ref{d*lemma}, we have
$$|K'/K|-1=\dd^*(K'/K)\leq \dd^*(K')-\dd^*(K)\leq\dd^*(G)-\dd^*(K).$$
Thus
Kneser's Theorem implies, in view of $w_ig\in L$ and $\gcd(w_i,\exp(H))=1$ (so that $\ord(w_ig)=\ord(g)$), that $$|\Sum{i=\dd^*(K)+1}{\dd^*(K')}\phi_K(w_i\cdot B')|=|\Sum{i=\dd^*(K)+1}{\dd^*(K')}\phi_K(w_i\cdot \{0,g\})|=|K'/K|,$$ and thus from (\ref{maxi-wo}) it follows that $$|\Sum{i=1}{\dd^*(K')}w_i\cdot B'|=|K'|,$$ contradicting the maximality of $K$, and completing the proof.
\end{proof}

\section{Theorems \ref{thm-ccd-cor-davstar} and \ref{thm-ccd-bonus}}

Theorems \ref{thm-ccd-cor-davstar} and \ref{thm-ccd-bonus} will be derived by an inductive argument from the following result. (Theorem \ref{ccd-weights} is easily derived from the proof of \cite{ccd} using the both the modifications mentioned in \cite{WEGZ} and those in \cite{hamconj};  see \cite{PhD-Dissertation} for a unified presentation of the arguments.)

\smallskip

\begin{theirtheorem}\label{ccd-weights} Let $G$ be an abelian group, let $S\in \Fc(G)$, let $S'|S$, and let $W=w_1\cdots w_n$ be a
sequence of integers such that $w_ig\neq 0$ for all $i$ and all
nonzero $g\in G$. Suppose $\h(S')\leq n\leq |S'|$.  Then there
exists $H\leq G$ and an $n$-setpartition $A=A_1,\ldots,A_n$ of a subsequence
$S''$ of $S$ such that $\Sum{i=1}{n}w_i\cdot A_i$ is $H$-periodic,
$|S'|=|S''|$, and
\be\label{duckbound}|\sum_{i=1}^n w_i\cdot A_i|\geq
((N-1)n+e+1)|H|,\ee where
$N=\frac{1}{|H|}|\bigcap_{i=1}^{n}(A_i+H)|$ and $e=\Sum{j=1}{n}(|A_j|-|A_j\cap \bigcap_{i=1}^{n}(A_i+H)|)$. Furthermore, if
$H$ is nontrivial, then $N\geq 1$ and
 $\supp({S''}^{-1}S)\subseteq \bigcap_{i=1}^{n}(A_i+H)$.
\end{theirtheorem}

\smallskip

The following basic result, which is a simple consequence of the pigeonhole principle, will be used in the proof \cite[Lemma 5.2.9]{Alfred-book}.

\smallskip

\begin{theirproposition}\label{mult_result} Let $G$ be an abelian group with $A,\,B\subseteq G$
 finite and nonempty. If $G$ is finite and $|A|+|B|\geq |G|+1$, then $A+B=G$.
\end{theirproposition}

\smallskip

We proceed with the proof of Theorems \ref{thm-ccd-cor-davstar} and \ref{thm-ccd-bonus} simultaneously.

\smallskip

\begin{proof}
Observe that the hypotheses of Theorem \ref{thm-ccd-cor-davstar} allow us to apply Theorem \ref{ccd-weights} with $G$, $S'|S$, $W=w_1\cdots w_n$ and $n$ the same in both theorems. Let $H$, $S''$, $A=A_1,\ldots,A_n$, $N$ and $e$ be as given by Theorem \ref{ccd-weights}. If $H$ is trivial, then (\ref{duckbound}) implies $|\Sum{i=1}{n} w_i\cdot A_i|\geq |S'|-n+1$, and if $H=G$, then $\Sum{i=1}{n} w_i\cdot A_i$ being $H$-periodic implies $|\Sum{i=1}{n} w_i\cdot A_i|=|G|=m$; in either case, (i) follows. Therefore we may assume $H$ is a proper, nontrivial subgroup. This completes the case when $|G|=m$ is prime in Theorem \ref{thm-ccd-cor-davstar}.

Concerning Theorem \ref{thm-ccd-bonus}(i), in view of $\h(S')\leq n$ and $|T^{-1}S|\geq n-\dd^*(G)+|S|-|S'|$, it is easily seen that the setpartition $B_1,\ldots, B_{\dd^*(G)}$ of $T$ can be extended to a setpartition $A_1,\ldots,A_n$ of a sequence $S''|S$, with $B_i\subseteq A_i$ for $i\leq \dd^*(G)$, $T|S''$ and $|S''|=|S'|$, by the following argument. Begin with $A_i=B_i$ for $i\leq \dd^*(G)$ and $A_i=\emptyset$ for $\dd^*(G)<i\leq n$. If $W|S$ are all terms with multiplicity at least $n$ and $W'=\prod_{g\in\supp(W)}g^n$, then augment the sets $A_i$ so that $\supp(W)\subseteq A_i$ for all $i$  (that is, simply include each $g\in \supp(W)$ in each set $A_i$ if it was not already there). We must have $|{W'}^{-1}W|\leq |S|-|S'|$, else it would have been impossible that a subsequence of $S$ with length $|S'|$ had an $n$-setpartition, which we know is the case since $\h(S')\leq n\leq |S'|$. All remaining terms in $T^{-1}{W}^{-1}S$ have multiplicity at most $n-1$, and so we can distribute all but $|S|-|S'|- |{W'}^{-1}W|$ of them among the $A_i$ so that no $A_i$ contains two equal terms, always choosing to place an element in an empty set if available. Since $|T^{-1}S|\geq n-\dd^*(G)+|S|-|S'|$, we are either assured that there are enough terms to fill all empty sets in this manner, or that we can move some of the terms from $W'$ (but not from $T$) placed in the $A_i$ with $i\leq \dd^*(G)$ so that this is the case, and then the resulting $A_i$ give the $n$-setpartition with the desired properties.

Consequently, (i) in Theorem \ref{thm-ccd-bonus} is trivial, and since the only nontrivial subgroup of $G$, when $|G|$ is prime, is $G$, we see that the case $|G|$ prime is complete for Theorem \ref{thm-ccd-bonus} as well.

We now proceed by induction on the number of prime factors of $m$.  We first show that (i) failing in Theorem \ref{thm-ccd-cor-davstar} implies the hypotheses of Theorem \ref{thm-ccd-bonus} hold (this is Claim B below), from which we infer that Theorem \ref{thm-ccd-bonus} implies Theorem \ref{thm-ccd-cor-davstar}. The remainder of the proof will then be devoted to proving Theorem \ref{thm-ccd-bonus} assuming by induction hypothesis that Theorem \ref{thm-ccd-cor-davstar} holds in any abelian group $G'$ with $|G'|$ having a smaller number of prime factors than $|G|$.

To this end, we assume (i) fails. Since (i) holds trivially when $n=1$ (in view of $n\geq \h(S')$), we may assume $n\geq 2$. Let $x:=|S|-|S'|\geq 0$. Since (i) fails, it follows from (\ref{duckbound}) that \ber\label{zap-cdt}((N-1)n+e+1)|H|&\leq& |S'|-n.\eer Much of the proof is contained in the following claim.

\medskip

\textbf{Claim B: } There exists a nontrivial subgroup $K$, $g'\in G$, and an $\dd^*(K)$-setpartition $B=B_1,\ldots, B_{\dd^*(K)}$ of a subsequence $T|S$ with $T\in \Fc(g'+K)$, such that \be\label{steally}\Sum{i=1}{\dd^*(K)}w_i\cdot B_i=\left(\Sum{i=1}{\dd^*(K)}w_i\right)g'+K\ee
and $T^{-1}S$ contains at least $n-\dd^*(K)+x$ terms from $g'+K$.

\begin{proof} There are two cases.

\smallskip

\emph{Case 1: } $N\geq 2$. If there does not exist $g'\in \bigcap_{i=1}^{n}(A_i+H)$ and $A_j$ and $A_k$ such that $j\neq k$ and \be\label{jump}|A_k\cap (g'+H)|+|A_j\cap (g'+H)|\geq |H|+1,\ee then it would follow from the pigeonhole principle (since $n\geq 2$) that $$|S'|=|S''|\leq \frac{1}{2}|H|Nn+e,$$ which combined with (\ref{zap-cdt}) implies $((N-1)n+e)|H|\leq \frac{1}{2}|H|Nn+e-n$, whence
$$|H|nN\leq 2(|H|-1)(n-e)\leq 2n(|H|-1),$$ implying $N<2$, a contradiction. Therefore we may assume such $g'$ and $A_j$ and $A_k$ exist, and w.l.o.g. $j=1$ and $k=2$. By translation we may also assume $g'=0$.

From Proposition \ref{mult_result}, (\ref{jump}) and $\gcd(w_i,\exp(G))=1$, it follows that \be\label{showsit}|w_1\cdot(A_1\cap H)+w_2\cdot(A_2\cap H)|=|H|.\ee Let $B_j=A_j\cap \bigcap_{i=1}^{n}(A_i+H)$ for $j=1,\ldots, n$, and note that $\phi_H(B_i)=\phi_H(B_j)$ for all $i$ and $j$. Let $K=H+\langle B_i\rangle$ and $T=\prod_{i=1}^{\dd^*(K)}B_i\in\Fc(K)$. From the conclusion of Theorem \ref{ccd-weights}, we know $T^{-1}S$ contains at least $n-\dd^*(K)+x$ terms from $K$ (since each $A_i$ intersects $\bigcap_{i=1}^{n}(A_i+H)$ in at least $N\geq 1$ points and $\supp({S''}^{-1}S)\subseteq \bigcap_{i=1}^{n}(A_i+H)$).

If
$\dd^*(H)\geq 2$, then from (\ref{showsit}) and $g'=0$ we find that  \be\label{spicder}H\subseteq \Sum{i=1}{\dd^*(H)}w_i\cdot B_i.\ee On the otherhand, if $\dd^*(H)=1$, then $|H|=2$, whence (\ref{jump}) and the pigeonhole principle imply that w.l.o.g. $|A_1\cap H|=|H|$, and thus (\ref{spicder}) holds in this case as well. Since $n\geq \dd^*(G)\geq \dd^*(K)$, it follows by Lemma \ref{d*lemma} that $$n-\dd^*(H)\geq \dd^*(K)-\dd^*(H)\geq \dd^*(K/H).$$ Thus, applying Lemma \ref{lem-split}, taking $\phi_H(B_i)$ for $A$ and $G/H$ for $G$ (recall that $g'=0$ and $|\phi_H(B_i)|=N\geq 2$), it follows that $$\Sum{i=\dd^*(H)+1}{\dd^*(K)}\phi_H(w_i\cdot B_i)=K/H,$$ which in view of (\ref{spicder}) implies that (\ref{steally}) holds. In view of the conclusion of the previous paragraph, this completes the claim.

\smallskip

\emph{Case 2: } $N=1$. Let $T$ be the subsequence of $S$ consisting of all terms from $g+H$, let $T'|T$ be the subsequence consisting of all terms with multiplicity at least $\dd^*(H)$, and let $B=\supp(T')$. From (\ref{zap-cdt}) and Theorem \ref{ccd-weights}, it follows that \be \label{lotsandlots} |T|\geq x+|S'|-e\geq (e+1)|H|+n+x-e\geq n+|H|+x.\ee By translation, we may w.l.o.g. assume $0\in \supp(T)$, and that $0\in \supp(T')$ if $\supp(T')\neq \emptyset$. We handle two subcases.

\medskip

\emph{Subcase 2.1: } Suppose there exists a subsequence $T_0|T$ with $\h(T_0)\leq \dd^*(H)$ and $|T_0|=\dd^*(H)+|H|-1$. Then we can apply the induction hypothesis to $T_0|T$ with $G$ taken to be $H$ and $n$ taken to be $\dd^*(H)$. Let $B=B_1,\ldots,B_{\dd^*(H)}$ be the resulting set partition and $T'_0$ the resulting subsequence of $T$. From (\ref{lotsandlots}), we see that \be\label{batmail}|{T'_0}^{-1}T|=|T|-|T'_0|=|T|-|T_0|=|T|-\dd^*(H)-|H|+1\geq n+x-\dd^*(H).\ee If (i) holds, then $|T_0|=\dd^*(H)+|H|-1$ implies that $$|\Sum{i=1}{\dd^*(H)}w_i\cdot B_i|=|H|,$$ and the claim is complete (in view of (\ref{batmail})) using $T'_0$ for $T$ and $H$ for $K$. On the otherhand, if (ii) holds with (say) subgroup $K\leq H$, $g'\in H$ and setpartition $B_1,\ldots,B_{\dd^*(H)}$, then (\ref{steally}) follows from (ii)(d) (taking $T$ to be $T''_0:=B_1\cdots B_{\dd^*(K)}$), while (ii)(a) and (\ref{lotsandlots}) imply ${T''_0}^{-1}T$ contains at least $$\dd^*(H)-\dd^*(K)+|T|-|T_0|=-\dd^*(K)+|T|-|H|+1\geq n-\dd^*(K)+x$$ terms from $g'+K$, whence the claim follows.

\medskip

\emph{Subcase 2.2:} There does not exist a subsequence $T_0|T$ with $\h(T_0)\leq \dd^*(H)$ and $|T_0|=\dd^*(H)+|H|-1$. Consequently, $$|\supp(T')|\dd^*(H)+|{T'}^{-1}T|\leq \dd^*(H)+|H|-2,$$ which, in view of (\ref{lotsandlots}), yields \be\label{indianajones}|T'|\geq n+x+2+(|\supp(T')|-1)\dd^*(H).\ee
Since $\vp_g(T')\leq \vp_g(T)\leq n+x$ for all $g\in G$ (in view of $\h(S')\leq n$), it follows that $|T'|\leq (n+x)|\supp(T')|$. Thus, in view of $n\geq \dd^*(G)\geq \dd^*(H)$ and $x\geq 0$, we conclude from (\ref{indianajones}) that $|\supp(T')|\geq 2$.

Let $K=\langle \supp(T')\rangle\leq H$ and let $T_0:=\prod_{g\in \supp(T')}g^{\dd^*(K)}$ be the subsequence of $T'$ (recall the definition of $T'$) obtained by taking each term with multiplicity exactly $\dd^*(K)\leq \dd^*(H)$. Observe, in view of (\ref{indianajones}) and $\dd^*(K)\leq \dd^*(H)$, that \ber\label{indy}|T_0^{-1}T'|&=&|T'|-|T_0|= |T'|-|\supp(T')|\dd^*(K)\\\nn&\geq& n+x+2+(|\supp(T')|-1)(\dd^*(H)-\dd^*(K))-\dd^*(K)\geq n+x-\dd^*(K).\eer Applying Lemma \ref{lem-split} with $A$ taken to be $\supp(T')$, we conclude (recall $0\in \supp(T')$) that $$\Sum{i=1}{\dd^*(K)}w_i\cdot B_i=K,$$ where $B_i=\supp(T')$ for $i=1,\ldots, \dd^*(K)$. Hence, in view of (\ref{indy}), we see that the claim follows (taking $T$ to be $T_0$).\end{proof}

Having now established Claim B, we see that it suffices to prove Theorem \ref{thm-ccd-bonus} to complete the inductive proofs of Theorems \ref{thm-ccd-cor-davstar} and \ref{thm-ccd-bonus}. Let $K$ be a maximal subgroup satisfying Claim B, and let $g'$, $T$ and $B_1,\ldots,B_{\dd^*(K)}$ be as given by Claim B. By translation we may w.l.o.g. assume $g'=0$. Let $S_0|S$ be the subsequence consisting of all terms $x$ with $\phi_K(x)\neq 0$, and let $e:=|S_0|$. As remarked earlier, if $K=G$, then Theorem \ref{thm-ccd-bonus}(i) follows trivially. Therefore assume $K<G$.
Observe that Claim B implies \be\label{cupidity} |T^{-1}S_0^{-1}S|\geq n-\dd^*(K)+x.\ee
\medskip

Suppose $\h(\phi_K(S_0))\geq \dd^*(G/K)$. Then let $g\in \supp(S_0)$ with $\vp_{\phi_K(g)}(\phi_K(S_0))\geq \dd^*(G/K)$ and let $L=K+\langle g\rangle$. By Lemma \ref{d*lemma}, we have \be\label{stucky2}\dd^*(L)\geq \dd^*(K)+\dd^*(L/K).\ee In view of (\ref{cupidity}), $\h(\phi_K(S_0))\geq \dd^*(G/K)\geq \dd^*(L/K)$ and $n\geq \dd^*(G)\geq \dd^*(L)$, we can find a subsequence $T'|T^{-1}S$ such that $\phi_K(T')=\phi_K(g)^{\dd^*(L/K)}0^{\dd^*(L)-\dd^*(K)}$, and thus such that $(TT')^{-1}S$ contains at least \be\label{stuckky}n-\dd^*(K)+x-(\dd^*(L)-\dd^*(K))= n-\dd^*(L)+x\ee terms from $L$. In view of (\ref{stucky2}), let $B_{\dd^*(K)+1},\ldots,B_{\dd^*(L)}$ be a setpartition of $T'$ such that $|B_i|=2$ and $\phi_K(B_i)=\{0,\phi_K(g)\}$, for $i=\dd^*(K)+1,\ldots,\dd^*(K)+\dd^*(L/K)$, and $|B_i|=1$ and $\phi_K(B_i)=\{0\}$, for $i=\dd^*(K)+\dd^*(L/K)+1,\ldots, \dd^*(L)$.

Applying Lemma \ref{lem-split} to $\{0,\phi_K(g)\}$, we conclude that $$|\Sum{i=\dd^*(K)+1}{\dd^*(K)+\dd^*(L/K)}\phi_K(w_i\cdot B_i)|=|L/K|,$$  and consequently (in view of (\ref{steally}) and (\ref{stucky2})) that $$|\Sum{i=1}{\dd^*(L)}w_i \cdot B_i|=|L|.$$ But now, in view also of (\ref{stuckky}), we see that the maximality of $K$ is contradicted by $L$. So we may instead assume $\h(\phi_K(S_0))< \dd^*(G/K)$.

\medskip

Let $R$ be a subsequence of $T^{-1}S$ such that $S_0|R$ and $|R|=|S_0|+\dd^*(G/K)$ (possible in view of (\ref{cupidity}), $x\geq 0$, $n\geq \dd^*(G)$ and Lemma \ref{d*lemma}). Moreover, from (\ref{cupidity}), \be\label{morelove} |(TR)^{-1}S|\geq n+x-\dd^*(K)-\dd^*(G/K),\ee
with all term of $(TR)^{-1}S$ contained in $K$ (since $S_0|R$).

Since $\h(\phi_K(S_0))< \dd^*(G/K)$, since $\phi_K(y)=0$ for $y|S_0^{-1}S$, and since $\phi_K(y)\neq 0$ for $y|S_0$, it follows that $\h(\phi_K(R))\leq \dd^*(G/K)$. Thus we can apply the induction hypothesis to $\phi_K(R)|\phi_K(R)0^{|G/K|-1}$ with $n=\dd^*(G/K)$ and $G$ taken to be $G/K$. Let $\phi_K(B_{\dd^*(K)+1}),\ldots,\phi_K(B_{\dd^*(K)+\dd^*(G/K)})$ be the resulting setpartition and $\phi_K(R')$ the resulting sequence, where $R'|R0^{|G/K|-1}$ and $B_{\dd^*(K)+1},\ldots,B_{\dd^*(K)+\dd^*(G/K)}$ is a setpartition of $R'$. Observe, since $\vp_0(\phi_K(R))=\dd^*(G/K)$, that $\supp(\phi_K(R')^{-1}\phi_K(R)0^{|G/K|-1})=\{0\}$, and thus that we can w.l.o.g. assume $R'=R$ and  likewise that $B_{\dd^*(K)+1},\ldots,B_{\dd^*(K)+\dd^*(G/K)}$ is a setpartition of $R$.

\medskip

Suppose (ii) holds and let $L/K$ be the corresponding subgroup. Since $\vp_0(\phi_K(R)0^{|G/K|-1})\geq |G/K|-1$, it follows in view of (ii)(c) that w.l.o.g. $g=0$ (where $g$ is as given by (ii)). But then (ii)(d) implies $$\Sum{i=\dd^*(K)+1}{\dd^*(K)+\dd^*(L/K)}w_{i}\cdot \phi_K(B_i)=L/K,$$ whence (\ref{steally}) implies $$\Sum{i=1}{\dd^*(K)+\dd^*(L/K)}w_{i}\cdot B_i=L.$$ In view of (ii)(a) and (\ref{morelove}), it follows that there are still at least \be\label{ghastlyghost}n+x-\dd^*(K)-\dd^*(G/K)+(\dd^*(G/K)-\dd^*(L/K))=n+x-\dd^*(K)-\dd^*(L/K)\ee terms remaining in $\left(\prod_{i=1}^{\dd^*(K)+\dd^*(L/K)}B_i\right)^{-1}S$ that are contained in $L$. Thus (in view of Lemma \ref{d*lemma}) by appending on an additional $\dd^*(L)-\dd^*(L/K)-\dd^*(K)\geq 0$ terms $B_i$, for $i=\dd^*(K)+\dd^*(L/K)+1,\ldots,\dd^*(L)$, with each such new $B_i$ consisting of a single element from $L$ contained in $\left(\prod_{i=1}^{\dd^*(K)+\dd^*(L/K)}B_i\right)^{-1}S$ (that is, $\supp(\prod_{i=\dd^*(K)+\dd^*(L/K)+1}^{\dd^*(L)}B_i)\subseteq L$ with $\prod_{i=\dd^*(K)+\dd^*(L/K)+1}^{\dd^*(L)}B_i|\left(\prod_{i=1}^{\dd^*(K)+\dd^*(L/K)}B_i\right)^{-1}S$), we see that $$\Sum{i=1}{\dd^*(L)}w_{i}\cdot B_i=L$$ and with $(\prod_{i=1}^{\dd^*(L)}B_i)^{-1}S$ containing at least (in view of (\ref{ghastlyghost})) $$n+x-\dd^*(K)-\dd^*(L/K)-(\dd^*(L)-\dd^*(L/K)-\dd^*(K))=n+x-\dd^*(L),$$ terms from $L$. Hence $L$ contradicts the maximality of $K$. So we may assume instead that (i) holds.

\medskip

As above, let $B_i$, for $i=\dd^*(K)+\dd^*(G/K)+1,\ldots, n$ (in view of (\ref{morelove})), be defined by partitioning, as singleton terms (i.e., $|B_i|=1$), $n-\dd^*(K)-\dd^*(G/K)$ of the terms of $\left(\prod_{i=1}^{\dd^*(K)+\dd^*(G/K)}B_i\right)^{-1}S=(TR)^{-1}S$ (which are all from $K$ in view of the comment after (\ref{morelove})).

If  \be\label{almostthere}\Sum{i=\dd^*(K)+1}{\dd^*(K)+\dd^*(G/K)}w_{i}\cdot \phi_K(B_i)=G/K,\ee then (\ref{steally}), $n\geq \dd^*(G)$ and Lemma \ref{d*lemma} imply that $$\Sum{i=1}{\dd^*(G)}w_{i}\cdot B_i=G.$$ Thus, in view of (\ref{morelove}), we see that  Claim B holds with $K=G$, contrary to assumption. Therefore we can assume (\ref{almostthere}) fails, which, in view of $|R|=|S_0|+\dd^*(G/K)$ and (i) holding for $\phi_K(R)$ with $n=\dd^*(G/K)$, implies that $e:=|S_0|\leq |G/K|-2$ and, in view of (\ref{steally}), that
 $$|\Sum{i=1}{n}w_{i}\cdot B_i|\geq (e+1)|K|.$$ The remaining conclusions for (ii) now follow easily from Claim B holding with $K$ (by the same arguments used for establishing Theorem \ref{thm-ccd-bonus}(i)), so that (ii) holds for $S'$ with subgroup $K$, as desired.
This completes the proof.
\end{proof}

\medskip

With the proof of Theorems \ref{thm-ccd-cor-davstar} and \ref{thm-ccd-bonus} complete, the improvement to Theorem \ref{thm-gao-coset-condition} follows as a simple corollary.

\begin{corollary}\label{cor-gao}Let $G$ be a finite abelian group, and let $S\in \Fc(G)$ with $|S|\geq |G|+\dd^*(G)$. Then either (i) $\Sigma_{|G|}(S)=G$ or (ii) there exists a coset $g+H$ such that all but at most $|G/H|-2$ terms of $S$ are from $g+H$.
\end{corollary}

\begin{proof} Let $|S|=|G|+\dd^*(G)+x$, so $x\geq 0$.
We assume (ii) fails for every $H$ and prove (i) holds. Note (ii) failing with $H$ trivial implies $\h(S)\leq \dd^*(G)+x+1$.

Suppose $\h(S)\leq \dd^*(G)+x$. Then we can apply Theorem \ref{thm-ccd-cor-davstar} with $n=\dd^*(G)+x$, $S=S'$ and $w_i=1$ for all $i$. If Theorem \ref{thm-ccd-cor-davstar}(ii) holds, then Theorem \ref{thm-ccd-cor-davstar}(ii)(c) implies Corollary \ref{cor-gao}(ii), contrary to assumption. If instead Theorem \ref{thm-ccd-cor-davstar}(i) holds, then from $|S|=|G|+\dd^*(G)+x$ we conclude that $\Sigma_{\dd^*(G)+x}(S)=G$. Since $\Sigma_{n}(S)=\sigma(S)-\Sigma_{|S|-n}(S)$ holds trivially for any $n$ (there is a natural correspondence between $S_0|S$ and ${S_0}^{-1}S|S$), it now follows that (i) holds for $S$, as desired. So we may assume $\h(S)= \dd^*(G)+x+1$.

By translation, we may w.l.o.g. assume $0$ is a term with multiplicity $\h(S)$ in $S$. We may also assume there is a nonzero $g\in G$ with $\vp_g(S)=\vp_0(S)=\h(S)$, else applying Theorem \ref{thm-ccd-cor-davstar} to $0^{-1}S|S$ completes the proof as in the previous paragraph. Let $S'|S$ be a maximal length subsequence with $\h(S')=\dd^*(G)+x$, let $A=\supp(S'^{-1}S)$, and let $K=\langle A\rangle$. Notice $\{0,g\}\subseteq A$. Hence, since $\h(S)=\dd^*(G)+x+1$, it follows from Lemma \ref{lem-split} that the hypotheses of Theorem \ref{thm-ccd-bonus} hold with $n=\dd^*(G)+x$, $S'|S$, $K$, and $w_i=1$ and $B_i=A$ for all $i$.
If Theorem \ref{thm-ccd-bonus}(i) holds, then $|G|=|\Sigma_{\dd^*(G)+x}(S)|=|\Sigma_{|G|}(S)|$ (as in the case $\h(S)\leq \dd^*(G)+x$), yielding (i). On the otherhand, Theorem \ref{thm-ccd-bonus}(ii) implies (ii) holds (in view of Theorem \ref{thm-ccd-cor-davstar}(ii)(c)). Thus the proof is complete.
%
\end{proof}

\smallskip

Next, the related corollary concerning Conjecture \ref{conjetura}. Note  the coset condition assumed below for $H$ trivial implies $\h(S)\leq |S|-|G|+1$, so the hypothesis $\h(S)\leq h\leq |S|-|G|+1$ is not vacuous. The case $h=|G|$ and $|S|=2|G|-1$ in Corollary \ref{cor-spud} is the result from  \cite{ordaz-quiroz}.

\smallskip

\begin{corollary}\label{cor-spud}Let $G$ be a finite abelian group, let $S\in \Fc(G)$, let $h\in\Z$ be such that  $\max\{\h(S),\,\dd^*(G)\}\leq h\leq |S|-|G|+1$, and let $W$ be a sequence of integers relatively prime to $\exp(G)$ with $|W|\geq h$. Suppose there does not exist a  coset $g+H$ such that all but at most $|G/H|-2$ terms of $S$ are from $g+H$. Then $\Sigma_{h}(W,S)=G$. In particular, $\Sigma(W,S)=G$
\end{corollary}

\begin{proof}The proof is identical to the case $\h(S)\leq \dd^*(G)+x$ in Corollary \ref{cor-gao} using $n=h$, the only other exception being that the identity $|\Sigma_{n}(W,S)|=|\Sigma_{|S|-n}(W,S)|$ is not necessarily valid for arbitrary $W$, $S$ and $n$, thus preventing the proof of Conjecture \ref{conjetura} itself.
\end{proof}

\smallskip

Now we derive Corollary \ref{thm-conj-Hamidoune-var} from Theorem \ref{thm-ccd-bonus}.

\smallskip

\begin{proof}Let $m=\exp(G)$. By considering $G$ as a $\Z/m\Z$-module (for notational convenience), we may w.l.o.g. consider $W$ as a sequence from $\Z/m\Z$, say w.l.o.g. $W=w_1\cdots w_n$, where $\ord(w_i)=m$ for $i\leq n-t$ (in view of the hypothesis $\gcd(w_i,\exp(G))=1$ for $w_i|{W'}^{-1}W$, where $|W'|=t$). Observe that we may assume $|S|=n+|G|-1$ and that there are distinct $x,\,y\in G$ with
$x^{n-t+1}y^{n-t+1}|S$, else Theorem \ref{thm-wegzi} implies the theorem (as if such is not the case, then in view of $\h(S)\leq n=|W|$ there would exist a $n$-setpartition of $S$ with $t$ sets of cardinality one). Since $\sigma(W)=0$, we may w.l.o.g. by translation
assume $x=0$.

Let $A\subseteq \supp(S)$ be all those elements with multiplicity at least $n-t$, let $K=\langle A\rangle$, let $R|S$ be the maximal subsequence with $\supp(R)=A$, let $T:=\prod_{g\in A}g^{\dd^*(K)}$, and let $T_0=\prod_{g\in A}g^{n-t}$. Notice $\{0,y\}\subseteq A$.
Hence, since $\h(S)\leq n$ and $|W|-t=n-t\geq \dd^*(G)$ by hypothesis, it follows from Lemma \ref{lem-split} applied to $A$ that the hypotheses of Theorem \ref{thm-ccd-bonus} hold with $n$ taken to be $n-t$, $B_i=A$ for $i=1,\ldots,\dd^*(K)$, and $S'=T_0(R^{-1}S)|S$.

If $|R|\leq |A|(n-t)+t$, then Theorem \ref{thm-wegzi} once more completes the proof (as then there exists an $n$-setpartition of $S$ with at least $t$ sets of cardinality one, in view of $\h(S)\leq n$). Therefore $|R|\geq |A|(n-t)+t+1$, and so \be\label{chowder}|S|-|S''|=|S|-|S'|=|R|-|T_0|\geq t+1,\ee where $S''$ is as given by Theorem \ref{thm-ccd-bonus}. Consequently, if Theorem \ref{thm-ccd-bonus}(i) holds, then $\Sigma_{n-t}({W'}^{-1}W,S'')=G$ with $|{S''}^{-1}S|\geq t$, whence $\Sigma_{n}(W,S)=\Sigma_{|W|}(W,S)=G$ follows, as desired. On the otherhand, if Theorem \ref{thm-ccd-bonus}(ii) holds, then Theorem \ref{thm-ccd-cor-davstar}(ii)(a)(d) implies $$\left(\Sum{i=1}{n-t}w_i\right)g+H\subseteq\Sum{i=1}{n-t}w_i\cdot A_i\subseteq \Sigma_{n-t}({W'}^{-1}W,S''),$$ where $g$, $H$ and the $A_i$ are as given by Theorem \ref{thm-ccd-cor-davstar}(ii), whence (\ref{chowder}), $\supp({S''}^{-1}S)\subseteq g+H$ (in view of (ii)(a)), and $\sigma(W)=0$ imply $$H=\left(\Sum{i=1}{n}w_i\right)g+H\subseteq \Sigma_{n}(W,S)=\Sigma_{|W|}(W,S),$$ as desired.
\end{proof}

\smallskip

Finally, we show Conjecture \ref{conjetura} holds when $\h(S)\geq \D(G)-1$. For this, we need the following modification of a result from \cite{Gao-M+D-1}. Note that the lemma is applicable for $W\in \Fc(\Z)$ as well, by applying it with the sequence $(w_1\cdot 1)(w_2\cdot 1)\cdots (w_n\cdot 1)$, where $W=w_1\cdots w_n$ and $1$ is the multiplicative identity of $R$.

\smallskip

\begin{lemma}\label{lemadavid}
Let $R$ be a ring of order $m$, let
$G$ be the additive group of $R$, and let $W,\,S\in \F(R)$ be nontrivial sequences with $|S|\geq |W|+\D(G)-1$. If $\vp_0(S)=\h(S)\geq
\D(G)-1$, then
$$\Sigma(W,S)=\Sigma_{|W|}(W,S).$$
\end{lemma}

\begin{proof}
Let $S'|S$ be the subsequence consisting of all nonzero
terms. Let $g\in \Sigma(W,S)$ be
arbitrary. Since $\Sigma_{|W|}(W,S)\subseteq \Sigma(W,S)$, we need to show that $g\in \Sigma_{|W|}(W,S)$.
%

If $g=0$ and ${\h(S)}\geq |W|$, then $0\in \Sigma_{|W|}(W,0^{\h(S)})\subseteq  \Sigma_{|W|}(W,S)$ (in view of $\vp_0(S)={\h(S)}$), as desired. If $g=0$ and ${\h(S)}\leq |W|-1$, then ${\h(S)}\geq \D(G)-1$ implies $|W|\geq \D(G)$, while $|S'|\geq |W|+\D(G)-1-{\h(S)}\geq \D(G)$. Thus \be\label{sendit} g \in \Sigma(W,S')\ee follows from the definition of $\D(G)$ applied to the sequence $(w_1s_1)(w_2s_2)\cdots (w_{\D(G)}s_{\D(G)})\in\Fc(R)$, where $w_1\cdots w_{\D(G)}|W$ and $s_1\cdots s_{\D(G)}|S'$. On the otherhand, if $g\neq 0$, then (\ref{sendit}) holds trivially. Thus we can assume (\ref{sendit}) regardless, and we choose $W_1|W$ and $S_1|S'$ such that $W_1=w_1\cdots w_t$, $S_1=s_1\cdots s_t$ and $g=\Sum{i=1}{t}w_is_i$, with $t$ maximal.

Note $t\leq |W|$. If $t\geq |W|-{\h(S)}$, then $g\in \Sigma_{|W|}(W,S_10^{\h(S)})\subseteq \Sigma_{|W|}(W,S)$, as desired. So we may
assume \be\label{teatime}t \leq |W|-{\h(S)}-1.\ee Hence
\be\label{howly-premp}|S_1^{-1}S'|\geq |W|+\D(G)-1-{\h(S)}-t\geq \D(G).\ee Observe, in view of
(\ref{teatime}) and the hypotheses, that
\be\label{howly}|W_1^{-1}W|=|W|-t\geq {\h(S)}+1\geq \D(G).\ee Let $S'=s_1\cdots s_ts_{t+1}\cdots s_{|S|-{\h(S)}}$ and $W=w_1\cdots w_tw_{t+1}\cdots w_n$. In view of (\ref{howly-premp}) and (\ref{howly}), let $$T:=(w_{t+1}s_{t+1})(w_{t+2}s_{t+2})\cdots (w_{t+\D(G)}s_{t+\D(G)})\in\Fc(R).$$
 Observe $|T|=\D(G)$, whence the definition of $\D(G)$ implies $T$ has a zero-sum subsequence, say (by re-indexing if necessary) $(w_{t+1}s_{t+1})(w_{t+2}s_{t+2})\cdots (w_{t+r}s_{t+r})$, where $r\geq 1$. But now the sequences $w_1\cdots w_{t+r}$ and $s_1\cdots s_{t+r}$ contradict the maximality of $t$, completing the proof.
\end{proof}

\smallskip

Note that Corollary \ref{conjetura-specialcase}(ii) failing with $H$ trivially implies $\h(S)\leq |G|$ for $|S|\leq 2|G|-1$, and that $2|G|-1\geq |G|+\D(G)-1$ in view of the trivial bound $\D(G)\leq |G|$ (see \cite{Alfred-book}). Thus the restriction $\h(S)\leq |G|$ in Corollary \ref{conjetura-specialcase} can be dropped when $|S|\leq 2|G|-1$, and thus, in particular, when $|S|=|G|+\D(G)-1$.

\smallskip

\begin{corollary} \label{conjetura-specialcase}
Let $G$ be a finite abelian group, and let $W\in \Fc(\Z)$ with $|W|=|G|$ and $\gcd(w_i,|G|)=1$ for all $w_i|W$. If $S\in \Fc(G)$ with  $|S|
\geq |G| +\D(G)-1$ and $|G|\geq \h(S)\geq \D(G)-1$, then either (i)  $\Sigma_{|G|}(W,S)=G$ or (ii) there exists a coset $g+H$ such that all but at most $|G/H|-2$ terms of $S$ are from $g+H$.
\end{corollary}

\begin{proof}We may w.l.o.g. assume $\vp_0(S)=\h(S)$. Thus our hypotheses allow us to apply Lemma \ref{lemadavid}, whence \be\label{ducks}\Sigma(W,S)=\Sigma_{|G|}(W,S).\ee Since we may assume (ii) fails with $H$ trivial, it follows that $\h(S)\leq |S|-|G|+1$. Consequently, since $|W|=|G|\geq \h(S)$, then the result follows from (\ref{ducks}) and Corollary \ref{cor-spud} applied with $h=\h(S)$ (in view of $\D(G)\geq \dd^*(G)+1$).\end{proof}

\textbf{Acknowledgements:}  We thank F. Kainrath, A. Geroldinger, and W. Schmid, for several helpful conversations regarding Section 4.

\end{document}